\documentclass[a4paper,12pt]{article} %

\usepackage{amsmath} %
\usepackage{amsthm} %
\usepackage{amssymb} %
\usepackage{mathabx} %
\usepackage{epsfig} %
\usepackage{psfrag} %
\usepackage[mathscr,mathcal]{eucal} %
\usepackage{enumerate} %
\usepackage{dsfont}

\theoremstyle{plain}

\newtheorem{theorem}{Theorem}[section]
\newtheorem{lemma}[theorem]{Lemma}
\newtheorem{corollary}[theorem]{Corollary}
\newtheorem{proposition}[theorem]{Proposition}

\theoremstyle{definition}

\newtheorem{definition}[theorem]{Definition}
\newtheorem{remark}[theorem]{Remark}

\numberwithin{equation}{section}

\newcommand{\N} {\mathbb N}
\newcommand{\Z} {\mathbb Z}

\newcommand{\R} {\mathbb R}

\def \ind{\mathds{1}}
\newcommand{\ER} {Erd\H{o}s-R\'enyi }
\newcommand{\cadlag} {c\`adl\`ag\ }

\newcommand{\ltwoord}{\ell_{2}^\downarrow}
\def \ltwopos {\ell_2^{+} }

\newcommand{\linford}{\ell_\infty^\downarrow}
\newcommand{\lzeroord}{\ell_0^\downarrow}

\newcommand{\prob}[1]{\ensuremath{\mathbf{P}\left(\,#1\,\right)}}
\newcommand{\expect}[1]{\ensuremath{\mathbf{E}\left(\,#1\,\right)}}

\newcommand{\condprob}[2]{\ensuremath{\mathbf{P}\left(\,#1\,\big|\,#2\,\right)}}
\newcommand{\condexpect}[2]{\ensuremath{\mathbf{E}\left(\,#1\,\middle|\,#2\,\right)}}

\def \dist {\mathrm{d}}

\newcommand{\bm}{\mathbf{m}}

\newcommand{\um}{\underline{m}}

\title
{Feller property of the multiplicative coalescent with linear deletion}

\begin{document}
\author { Bal\'azs R\'ath\textsuperscript{1}}

%\footnotetext[1]{University of Oxford, United Kingdom. E-mail: martin@stats.ox.ac.uk}
\footnotetext[1]{MTA-BME Stochastics Research Group, Hungary. E-mail: rathb@math.bme.hu}

\maketitle

\begin{abstract}
We modify the definition of Aldous' multiplicative coalescent process \cite{aldous_excursions_coalescent} 
and 
 introduce the multiplicative coalescent with linear deletion (MCLD). A state of this
 process is a square-summable decreasing sequence of cluster sizes. Pairs of clusters
 merge with a rate equal to the product of their sizes and clusters are deleted with a rate linearly 
 proportional to their size. We prove that the MCLD is a Feller process.
 This result is a key ingredient in  the description of  scaling limits of the evolution of component sizes 
 of the mean field frozen percolation model \cite{br_frozen_2009} and the so-called rigid 
 representation of such scaling limits \cite{rigid_paper}.

\bigskip

\medskip


\end{abstract}

\section{Introduction}

Let us define 
\begin{align*}
\linford &= \{ \;\um=(m_1, m_2, \dots) \; : \;
m_1\geq m_2\geq \dots \geq 0 \; \}, \\
\ltwoord &= \{ \; \um \in \linford \; : \; \sum_{i=1}^\infty m_i^2<\infty \},
 \\
\lzeroord &= \{ \; \um \in \linford \; : \; \exists \; i_0 \in \N \; : 
\; m_i=0 \; \text{ for any } \; i \geq i_0 \; \}.
\end{align*}

For $\um,\um' \in \ltwoord$ one defines the distance
\begin{equation}\label{eq_def_d_metric}
 \dist(\um ,\um')=\Vert \um-\um' \Vert_2= \left( \sum_{i \geq 1} (m_i-m'_i)^2  \right)^{1/2}.
 \end{equation}
The metric space $\left( \ltwoord, \, \dist(\cdot,\cdot) \right)$ is complete and separable. 

The multiplicative coalescent process (or briefly $\mathrm{MC}$ process), defined in 
\cite[Section 1.5]{aldous_excursions_coalescent}, is
a continuous-time Markov process $\bm_t, t \geq 0$ with state space $\ltwoord$. 
The state  $\bm_t$ represents the ordered sequence of sizes of components, 
where two components of size $m_i$ and $m_j$ merge with rate $m_i \cdot m_j$.
By \cite[Proposition 5]{aldous_excursions_coalescent}, the multiplicative coalescent process
has the Feller property with respect to the metric $\dist(\cdot,\cdot)$ on $\ltwoord$.
On the other hand, if $\bm_0 \in \linford \setminus \ltwoord$, then all of the components instantaneously
coagulate and form one component with infinite mass, see \cite[Section 2.1]{limic_phd_thesis}.
In Section \ref{subsection:mc}, we collect  the basic results about MC relevant for our study.

\medskip

Let $\lambda \in \R_+$. For any $\um \in \ltwoord$
we want to define a continuous time Markov process $\bm_t$ with state space $\ltwoord$ 
where $\bm_0=\um$ and $\bm_t$ represents the ordered sequence of sizes of components of a 
coagulation-deletion
process
at time $t$. We want the dynamics of the process $\bm_t, t \geq 0$ to satisfy 
\begin{equation}\label{mcld_informal_def}
\begin{array}{l}
 \text{(i) two components of size $m_i$ and $m_j$ merge with rate $m_i \cdot m_j$,} \\
 \text{(ii) a component of size $m_i$ is deleted with rate $\lambda \cdot m_i$.}
 \end{array}
\end{equation}

We are going to call such a process a \emph{multiplicative coalescent with linear deletion} with deletion rate $\lambda$, and briefly denote it by $\mathrm{MCLD}(\lambda)$.

\medskip

If $\um \in \lzeroord$  then the $\mathrm{MCLD}(\lambda)$ 
process obviously exists and $\bm_t \in \lzeroord $  for any $t \geq 0$. In fact, if
$\um \in \linford$ with $\sum_{i=1}^\infty m_i<\infty$ then the definition of $\mathrm{MCLD}(\lambda)$
is still quite simple because the time between consecutive coalescences/deletions is always positive.
On the other hand, for initial conditions with infinite total mass, the set of times when a coalescence or deletion occurs will be dense in $\R_+$, and it is not a priori clear that a well-defined stochastic process satisfying \eqref{mcld_informal_def}
exists (see Remark \ref{two_dim_frozen_percolation_non_existence} below for  related non-existence results).
 
In Section \ref{subsection:deletions} we will give a \emph{graphical construction} of the process $\bm_t$  with initial state
 $\um \in \ltwoord$ and deletion rate $\lambda$. 
 This construction of $\mathrm{MCLD}(\lambda)$ is similar to, but not as simple as the graphical construction of the $\mathrm{MC}$
 given in \cite[Section 1.5]{aldous_excursions_coalescent} because $\mathrm{MCLD}(\lambda)$ lacks the monotonicity properties
 of $\mathrm{MC}$, see Remark \ref{remark_lack_of_monotonicity} below.
 In Section \ref{subsection:deletions} we also prove the following proposition.
 
\begin{proposition}\label{lemma_mcld_graphical_rep_cadlag}
For any  $\um \in \ltwoord$  our graphical construction  of $\mathrm{MCLD}(\lambda)$ (see Section \ref{subsection:deletions})
almost surely gives a function $t \mapsto \bm_t$ with
$\bm_0=\um$ which 
is \cadlag with respect to the  $\dist(\cdot,\cdot)$-metric.
\end{proposition} 
 
 The main result of this paper is that our construction indeed gives rise to a well-behaved continuous-time Markov process on 
$\ltwoord$:

\begin{theorem}[Feller property]
\label{thm:feller_basic}
Let $\um^{(n)}, n \in \N$ be a convergent sequence of elements of  $\ltwoord$ with limit $\um^{(\infty)}$,
 i.e., $\lim_{n \to \infty} \dist(\um^{(n)},\um^{(\infty)})=0 $.
For any $t \in \R_+$ and $n \in \N_+ \cup \{\infty\}$, denote by $\bm^{(n)}_t$ the $\mathrm{MCLD}(\lambda)$ process with initial condition 
$\um^{(n)}$ at time $t$. For any $ t \geq 0$ we have
\begin{equation}\label{feller_convergence_in_distribution}
\bm^{(n)}_t  \stackrel{d}{\longrightarrow}   \bm^{(\infty)}_t , \quad n \to \infty,
\end{equation}
where $\stackrel{d}{\longrightarrow}$ denotes convergence in distribution of random variables on the 
Polish space $(\ltwoord, \dist(\cdot,\cdot))$.
\end{theorem}
We will prove Theorem \ref{thm:feller_basic} in Section \ref{section:feller_coupling_proof} using an argument
that involves truncation and coupling.

\begin{remark}\label{remark_lack_of_monotonicity}
 The reason why the proof of the Feller property for $\mathrm{MCLD}(\lambda)$ is more involved than the proof of the Feller property for $\mathrm{MC}$ (c.f.\ the proof of \cite[Proposition 5]{aldous_excursions_coalescent} in 
 \cite[Section 4.2]{aldous_excursions_coalescent})   is that the  natural graphical construction of $\mathrm{MCLD}(\lambda)$ is not \emph{monotone}:
 
  If we obtain $\bm'_t, t \geq 0$ from $\bm_t, t \geq 0$ by inserting an extra deletion event at time $t_1$ then it might happen that this deletion prevents later coagulations and deletions, so that $\bm'_{t_2}$ has more/bigger components than
 $\bm_{t_2}$ for some $t_2>t_1$. Similarly,  insertion of an extra coagulation event at some time might lead to the deletion of more/bigger components and thus create a state with fewer/smaller components at a later time.
\end{remark}

\subsection{Motivation, related results}

Our reason for developing the theory of $\mathrm{MCLD}(\lambda)$ on the state space $\ltwoord$
is that we want to understand the scaling limit of the time evolution of large connected component sizes in the
self-organized critical mean field frozen percolation model \cite{br_frozen_2009}, as we now explain.

\medskip

The frozen percolation process on the binary tree was defined in \cite{Aldousfrozen}: the model is a modification of the dynamical percolation process on the binary tree which makes the following informal description precise: edges appear with rate $1$ and if an infinite component appears, we immediately ``freeze" it, and we do not allow edges with an end-vertex in a frozen component to appear.
  
\begin{remark}\label{two_dim_frozen_percolation_non_existence}
I. Benjamini and O. Schramm  showed that it is impossible to
define a similar modification of the percolation process on
$\Z^2$, c.f.\ \cite[Section 3, Remark (i)]{berg_toth}.
Various modifications of the two-dimensional frozen percolation model where large finite clusters are frozen
are further explored in \cite{berg_nolin, kiss_forzen, berg_nolin_exceptional_scales, berg_kiss_nolin_prevalence}.
 The result of \cite{kiss_manulescu_sid} about the closely related
model of two-dimensional
self-destructive percolation implies non-existence of the so-called two dimensional forest fire process, 
c.f.\  \cite[Section 3.2]{kiss_manulescu_sid}. However, the result of \cite{seven_dim_forest_fires} about
self-destructive percolation on the high-dimensional lattice $\Z^d$ indicates that the self-organized critical
 forest fire process
  on $\Z^d$ should exist if $d$ is high enough. The existence and uniqueness of the subcritical forest fire process on $\Z^d$
  was proved in \cite{durre_1, durre_2}.
 \end{remark}

Let us now recall the notion of \emph{mean-field frozen percolation process} from \cite{br_frozen_2009} (using slightly different notation). 
 
 \begin{definition}[$\mathrm{FP}(n,\lambda(n))$]
\label{def_frozen_percolation_model}
We start with a graph $F^{(n)}_0$ on $n$ vertices. Between each pair of unconnected vertices an edge appears with rate $1/n$; also,
every connected component of size $k$ is deleted with rate $\lambda(n)\cdot k$. (When a component is deleted, its vertices as well as its 
edges are removed from the graph.) 
Let $F^{(n)}_t$ be the graph at time $t$.  
Denote by 
\[ 
\mathbf{M}^{(n)}(t)= \left(M_1^{(n)}(t), M_2^{(n)}(t), \dots \right) \in \lzeroord 
\] 
the sequence of component sizes of $F^{(n)}_t$, arranged in decreasing order. 
\end{definition} 
    
Then $\mathbf{M}^{(n)}(t), t \geq 0$ is a Markov process -- let us call it here the
frozen percolation component process on $n$ vertices with lightning rate $\lambda(n)$,
or briefly $\mathrm{FP}(n,\lambda(n))$.  In fact,  up to time-change, $\mathbf{M}^{(n)}(t), t \geq 0$ evolves according to the rules \eqref{mcld_informal_def} of $\mathrm{MCLD}$.

\begin{remark}\label{remark_forest}
 We note that $\mathrm{FP}(n,\lambda(n))$  is a simplification of the mean field forest fire model \cite{br_bt_forest}, the definition of which agrees with Definition \ref{def_frozen_percolation_model} above, with the only difference that in the forest fire model we only delete the edges of the connected components that are destroyed by fire, i.e., a destroyed component of size $k$ is immediately replaced by $k$ singletons. The mean field forest fire model behaves very similarly to the mean field frozen percolation model
(e.g., the self-organized critical behaviour of the two models are quite similar, see also Remark \ref{remark_future_window_process}\eqref{remark_future_forest} below), however the mathematics of the
  mean field frozen percolation model is simpler than that of the mean field forest fire model, e.g.\
  the solution of the system of differential equations that appears in 
  \cite[Theorem 1.2]{br_frozen_2009} is fairly explicit, while the system of differential equations
  that appears in \cite[Theorem 2]{br_bt_forest} currently does not have an explicit solution.
 \end{remark} 
 
\begin{remark} One studies the asymptotic behaviour of the component size structure of $\mathrm{FP}(n,\lambda(n))$ when $1 \ll n$ and
 $1/n \ll \lambda(n) \ll 1$.
Definition \ref{def_frozen_percolation_model} above is slightly different from the one 
proposed in \cite[Section 5.5]{Aldousfrozen} and studied in \cite{merle_normand} where connected components are frozen when their size exceeds a threshold  $\omega(n)$ satisfying
 $1 \ll \omega(n) \ll n$. The results \cite[Theorem 1.1]{merle_normand} and 
 \cite[Theorem 1.2]{br_frozen_2009} are very similar: indeed, if one is interested in small connected component densities then the
 two models produce exactly the same (self-organized critical) behaviour. However, if one is interested in the scaling limit of big component dynamics, 
 the exact deletion mechanism does crucially enter the picture. 
 \end{remark}

   We are interested in identifying the scaling limit of $\mathrm{FP}(n,\lambda(n))$
  as $n \to \infty$.
In order to describe the kind of result we are after, let us recall that 
 the large components of the dynamical \ER random graph process 
 in the critical window $\mathcal{G}(n,\frac{1+t n^{-1/3}}{n}), t \in \R$, scaled by $n^{2/3}$,
converge in law to the \emph{standard multiplicative coalescent} process $(\mathcal{M}(t), t \in \R)$, see 
\cite[Section 4.3]{aldous_excursions_coalescent}.

\begin{remark}
The family of multiplicative coalescent processes defined for all $t \in \R$
 (i.e., the \emph{eternal} $\mathrm{MC}$ processes) are characterized in \cite{aldous_limic}.
  The class of  inhomogeneous random graph models 
 whose scaling limit is the standard $\mathrm{MC}$ is  explored in  \cite{bh_ho_leeu, basin} (see also references therein).
The scaling limits of other classes of inhomogeneous random graph models are related to non-standard eternal $\mathrm{MC}$ processes, see  \cite{aldous_limic,bh_novel}. 
The continuum scaling limit of the metric structure of critical random graphs is studied in
\cite{goldschmidt_et_al, bh_h_ss_metric} (see also references therein).
\end{remark}

The next result gives a scaling limit for 
the frozen percolation process started from a critical \ER graph.

\begin{proposition}\label{prop:FPlimit}
Fix $u \in \R$ and let $F^{(n)}_0$ be an \ER graph $\mathcal{G}(n,p)$ with edge probability $p=\frac{1+u n^{-1/3}}{n}$.
Let $\lambda>0$ and let
$\mathbf{M}^{(n)}(t), t \geq 0$ be the 
$\mathrm{FP}(n,\lambda n^{-1/3})$ process 
with initial state $F^{(n)}_0$.
Define $\mathbf{m}^{(n)}(t), t \geq 0$ by 
\begin{equation}\label{frozen_mcld_scaling}
 \mathbf{m}^{(n)}(t):=
\left( n^{-2/3} M_1^{(n)}(n^{-1/3}t), n^{-2/3}M_2^{(n)}(n^{-1/3} t), \dots \right). 
\end{equation}

Then as $n\to\infty$ the finite dimensional marginals of
the sequence of $\ltwoord$-valued processes $\mathbf{m}^{(n)}(t), t \geq 0$ converge in law to the
finite dimensional marginals of the
$\mathrm{MCLD}(\lambda)$ process $(\mathbf{m}(t), t \geq 0)$ started from an initial state with distribution $\mathbf{m}(0) \sim \mathcal{M}(u)$ (i.e., the state of the standard multiplicative coalescent process at time $u$), i.e., for every $k \in \N$ and $0 \leq t_1< t_2<\dots< t_k$ we have
\begin{equation*}
 \Big( \mathbf{m}^{(n)}(t_1), \mathbf{m}^{(n)}(t_2), \dots, \mathbf{m}^{(n)}(t_k) \Big)
 \stackrel{\mathrm{d}}{\longrightarrow} 
\Big( \mathbf{m}(t_1), \mathbf{m}(t_2), \dots, \mathbf{m}(t_k) \Big), \quad n \to \infty.
\end{equation*}
\end{proposition}
The proof of Proposition \ref{prop:FPlimit} follows as an application of Theorem \ref{thm:feller_basic}
(for details of the proof, we refer to \cite[Proposition 6.10]{rigid_paper}).

\begin{remark}\label{remark_future_window_process}
$ $

\begin{enumerate}[(i)]
\item Loosely speaking, if $(\mathbf{m}(t), t \geq 0)$ is the
$\mathrm{MCLD}(\lambda)$ process started from an initial state with distribution $\mathbf{m}(0) \sim \mathcal{M}(u)$ (this is the limit object that appears in Proposition \ref{prop:FPlimit}), then we have $\mathbf{m}(t) \sim \mathcal{M}(u+t-\Phi(t))$, where $\Phi(t)$ denotes the sum of the sizes of the components deleted up to time $t$ (see \cite[Proposition 6.7(ii)]{rigid_paper} for a precise formulation of this property).
In fact, in  \cite[Proposition 6.7]{rigid_paper} we give a 
representation of $(\mathbf{m}(t), t \geq 0)$ on the probability space of a standard Brownian motion
using what we call the ``rigid'' representation of $\mathrm{MCLD}(\lambda)$. We  note that Theorem \ref{thm:feller_basic}
is also  crucially used when we extend our rigid representation results from $\lzeroord$ to $\ltwoord$ in 
\cite[Section 5]{rigid_paper}.

\item In \cite{james_balazs_window_process} we describe the possible scaling limits
that can arise from a $\mathrm{FP}(n,\lambda n^{-1/3})$ process started from an empty graph.
 The possible limit objects are eternal $\mathrm{MCLD}(\lambda)$ processes (i.e., they are defined for any $t \in \R$).
The ``arrival at the critical window'' gives rise to a non-stationary $\mathrm{MCLD}(\lambda)$ scaling limit,
while the scaling limit in the ``self-organized critical'' regime is a stationary $\mathrm{MCLD}(\lambda)$
(see also \cite[Remark 6.8]{rigid_paper}).

\item \label{remark_future_forest} We conjecture that the scaling limit of the coagulation-fragmentation dynamics of big components of the mean field forest fire model (c.f.\ Remark \ref{remark_forest} above) with lightning rate $\lambda n^{-1/3}$ is also an 
 $\mathrm{MCLD}(\lambda)$ process.
\end{enumerate}
\end{remark}

\section{Notation and basic results}
\label{subsection:mc}

The aim of this section is to collect some basic results about the multiplicative coalescent from 
\cite{aldous_excursions_coalescent} and \cite{limic_phd_thesis}. In some cases, we will augment these results to fit our purposes
or present them using different notation.

\medskip

We define 
\begin{equation*} 
\ltwopos = \left\{ x = (x_1,x_2,\dots) \, : \, \forall  i \; \; x_i \geq 0,  \quad
 \sum_{i \geq 1} x_i^2 < + \infty \right\}.
\end{equation*}
We have $\ltwoord \subseteq \ltwopos$. 
Define the mapping 
\begin{equation} \label{eq:def_ord}
\mathrm{ord}: \ltwopos \to \ltwoord
\end{equation} 
by letting $\mathrm{ord}(\underline{x})$
 be the decreasing rearrangement of $\underline{x} \in \ltwopos$.

\begin{definition}\label{def_weights_from_graph}
If $\um \in \ltwoord$ and $G$ is a graph with vertex set $V \subseteq \N_+$, denote by
${\mathrm{ord}}(\um,G)$ the ordered sequence of the weights of the connected components of $G$.
More precisely, if $\mathcal{C}_1, \mathcal{C}_2, \dots$ is the sequence of the vertex sets
of the connected components of $G$, we define
\begin{equation}\label{eq_def_ord_um_G}
 \underline{x}_G = \left( \sum_{i \in \mathcal{C}_1} m_i, \sum_{i \in \mathcal{C}_2} m_i, \dots \right) \quad 
\text{ and}
\quad \mathrm{ord}(\um,G)\stackrel{\eqref{eq:def_ord}}{=} {\mathrm{ord}}(\underline{x}_G),
 \end{equation}
assuming that $\underline{x}_G \in \ltwopos$. 
We also denote
\begin{equation}\label{S_2_G}
 S_2^{G} = \sum_{k=1}^{\infty} \left(\sum_{j \in \mathcal{C}_k} m_j \right)^2=
 \Vert x_G \Vert_2^2
 =\Vert {\mathrm{ord}}(\underline{x}_G)\Vert_2^2. 
 \end{equation}
\end{definition}

Let us now state an elementary yet useful result which involves the metric  $\dist(\cdot,\cdot)$
defined in \eqref{eq_def_d_metric}.

\begin{lemma}\label{lemma_dominate_graph_compare_aldous}
If $\um \in \ltwoord$ and $G,G'$ are graphs with vertex sets $V,V' \subseteq \N_+$ such that
$V \subseteq V'$, $G \subseteq G'$ and $\mathrm{ord}(\um,G) \in \ltwoord$  then we have
\begin{equation*}
\dist\left({\mathrm{ord}}(\um,G), {\mathrm{ord}}(\um,G') \right) 
\leq \sqrt{ \Vert{\mathrm{ord}}(\um,G')\Vert_2^2 - \Vert{\mathrm{ord}}(\um,G)\Vert_2^2  }.
\end{equation*}
\end{lemma}
\begin{proof}
This is a special case of \cite[Lemma 17]{aldous_excursions_coalescent}.

\end{proof}

Let us recall the graphical construction used in \cite[Section 1.5]{aldous_excursions_coalescent} to define
the multiplicative coalescent process.

\begin{definition}\label{def_mc_graphical}
 Let $\left( \xi_{i,j} \right)_{1\leq i<j <\infty} $ denote independent random variables with 
$\mathrm{EXP}(1)$ distribution. Given $\underline{x} \in \ltwopos$
let us define the simple graph $G_t$ with vertex set $\N_+$ and an edge between $i$ and $j$ if and only if $\xi_{i,j} \leq t x_i x_j$. 
For $i,j \in \N_+$ we denote by $i \stackrel{G_t}{\longleftrightarrow} j$ the event that $i$ and $j$ are connected 
by a simple path
in the graph $G_t$.
\end{definition}
Given $G_t$ we define the connected components
$\left( \mathcal{C}_k(t) \right)_{k=1}^{\infty}$ of $G_t$ by
\begin{equation}\label{eq_def_componnets_of_G_t}
 i_k= \min\{\, \N_+ \setminus \cup_{l=1}^{k-1} \mathcal{C}_l(t)\, \}, \quad 
  \mathcal{C}_k(t)=\{\, i \in \N_+ \; : \; 
i \stackrel{G_t}{\longleftrightarrow} i_k   \, \}, \quad k \geq 1. 
\end{equation}

Note that we have
\begin{equation}\label{S2_def}
S_2^{G_t} \stackrel{ \eqref{S_2_G} }{=} 
S_2^{G_0} + \sum_{i \neq j} x_i x_j  
\ind[ i \stackrel{G_t}{\longleftrightarrow} j ] 
\end{equation}
and $S_2^{G_0}= \sum_{i=1}^{\infty} x_i^2 < +\infty$ if $\underline{x} \in \ltwopos$.

The statement of the next lemma follows from \cite[Proposition 5]{aldous_excursions_coalescent} and
shows that 
 Definitions \ref{def_weights_from_graph} and \ref{def_mc_graphical} give rise to
  a graphical representation of the $\ltwoord$-valued
 multiplicative coalescent process  with initial state
 $\um \in \ltwoord$ in the form $\mathrm{ord}(\um,G_t), t \geq 0$.

\begin{lemma}\label{lemma:S_2_G_t_as_finite}
 For any $t \geq 0$ and $\underline{x} \in \ltwopos$ we have
\begin{equation}\label{S_2_G_t_as_finite}
\prob{S_2^{G_t}<+\infty}=1.
\end{equation}
In particular, for any $t \in \R_+$ the weights of the connected components of $G_t$ are almost surely finite:
\begin{equation}\label{finite_components}
\prob{ \forall \, k \in \N_+ \; : \; \sum_{i \in \mathcal{C}_k(t)} x_i < +\infty  }=1.
\end{equation}
\end{lemma}

%\begin{proof}
%Given $\underline{x} \in \ltwopos$ and $t \geq 0$
%we first choose $m$ big enough so that 
%$S_2^{G^{m \uparrow}_0}= \sum_{i=m+1}^{\infty} x_i^2  \leq \frac{1}{2t}$ holds.
% Then we apply Corollary \ref{lemma_small_S_2_does_not_increase_much} to deduce
%$\expect{S_2^{G^{m \uparrow}_t} }< +\infty$.
%Now if we condition on the component sizes of 
%$G^{m \downarrow}_t$   and  $G^{m \uparrow}_t$, we can apply  
% Lemma \ref{lemma_bipartite} to construct the graph $G_t$ as $B_t$ and
%use \eqref{eq_bipartite_S_2_B_t_statement_finite_I} to deduce that 
%$S_2^{G_t}$ is almost surely finite.
%\end{proof}

\medskip

The next lemma is an extended version of \cite[(2.2)]{limic_phd_thesis}.
\begin{lemma}
For any $\underline{x} \in \ltwopos$ and $i,j \in \N_+$ and $t  < \frac{1}{ S_2^{G_0}} $ we have
\begin{equation}\label{eq_i_j_conn_in_G_t}
\prob{ i \stackrel{G_t}{\longleftrightarrow} j } \leq 
  \frac{ x_i \cdot x_j \cdot t }{1-t \cdot S_2^{G_0}} .
\end{equation}
\end{lemma}

\begin{proof}
\begin{multline}\label{eq_proof_connect_i_j_in_G_t}
\prob{ i \stackrel{G_t}{\longleftrightarrow} j } \leq \\
\sum_{k=1}^{\infty} 
\mathbf{P}\left(
\begin{array}{c}
\exists\,  i_0,\dots,i_k \in \N_+ \; : \;  i_0=i,\, i_k=j \text{ and } \\
(i_0, i_1, \dots, i_{k-1}, i_k) \text{ is a simple path in } G_t
\end{array}
\right) \leq \\
 \sum_{k=1}^{\infty} \;  \sum_{ (i_1,\dots,i_{k-1}) \in \N_+^{k-1}} \;
\prod_{l=1}^{k} (1-\exp(-x_{i_{l-1}} x_{i_l} t) ) \leq \\
\sum_{k=1}^{\infty} \;  \sum_{ (i_1,\dots,i_{k-1}) \in \N_+^{k-1}} \;
\prod_{l=1}^{k}  x_{i_{l-1}} x_{i_l} t = 
x_i x_j t \cdot  \sum_{k=1}^{\infty} 
\sum_{ (i_1,\dots,i_{k-1}) \in \N_+^{k-1}} \;
\prod_{l=1}^{k-1}   x_{i_l}^2 t  =\\
x_i x_j t \cdot  \sum_{k=1}^{\infty} (t \cdot S_2^{G_0})^{k-1} = 
\frac{ x_i \cdot x_j \cdot t }{1-t \cdot S_2^{G_0}}.
\end{multline}
\end{proof}

\begin{corollary}\label{lemma_small_S_2_does_not_increase_much}
 For any $\underline{x} \in \ltwopos$, $t \geq 0$ and $i,j \in \N_+$, if
\begin{equation}\label{eq_bound_on_initial_S_2}
S_2^{G_0} \leq \frac{1}{2t}
\end{equation}
 holds then we have
\begin{equation}\label{small_S2_if_small_initial_S2}
\expect{S_2^{G_t} } \leq 2 S_2^{G_0} 
\end{equation}
\end{corollary}
\begin{proof} Using \eqref{S2_def}, \eqref{eq_i_j_conn_in_G_t} and \eqref{eq_bound_on_initial_S_2} we obtain
\begin{multline}\label{eq_proof_expect_control_S_2_G_t}
\expect{S_2^{G_t} } 
%S_2^{G_0} + \sum_{i \neq j} x_i x_j  
%\prob{ i \stackrel{G_t}{\longleftrightarrow} j } 
\leq
S_2^{G_0} +2t \sum_{i\neq j} x_i^2 x_j^2\leq S_2^{G_0} +2t \cdot (S_2^{G_0})^2
\stackrel{\eqref{eq_bound_on_initial_S_2}}{\leq} 2 S_2^{G_0}.
\end{multline}
\end{proof}

The  next lemma is based on \cite[Lemma 23]{aldous_excursions_coalescent}
and \cite[(2.5)]{limic_phd_thesis}. It will be used in Section \ref{section:feller_coupling_proof}
to show that the truncated process is close to the original process if the truncation threshold is chosen big enough.

\begin{lemma}\label{lemma_bipartite}
Let $\underline{x},\underline{y} \in \ltwopos$ and $t \geq 0$. 
Denote the index set of $\underline{x}$ by $I$ and the index set of $\underline{y}$ by $J$.
Denote by 
\[a= \Vert \underline{x} \Vert_2^2<+\infty \qquad \text{and} \qquad
 b = \Vert \underline{y} \Vert_2^2<+\infty.\]
 Consider the bipartite random graph $B_t$ with vertex set $I \cup J$, where
$i \in I$ and $j \in J$ are connected with probability $1-\exp(-t x_i y_j)$.
Then we have
\begin{equation}\label{eq_bipartite_S_2_B_t_statement_finite_I}
|I| < +\infty \quad \implies \quad \expect{ S_2^{B_t}} < +\infty .
\end{equation}
Moreover, if \begin{equation}\label{eq_assumption_bipartite}
t^2 a  b \leq \frac12,
\end{equation}
 holds then we have
  \begin{equation}\label{eq_bipartite_S_2_B_t_statement}
  \expect{ S_2^{B_t}} - a \leq 2 b \cdot
   \left( 1+ t  a \right)^2.
  \end{equation}
   \end{lemma}
   
\begin{proof} 
First note that, similarly to  \eqref{S2_def}, we have
\begin{multline}\label{eq_bipartite_S_2_expand}
\expect{ S_2^{B_t}}= a + b+
\sum_{i_1 \neq i_2 \in I} x_{i_1} x_{i_2} \prob{ i_1 \stackrel{B_t}{\longleftrightarrow} i_2 }
+ \\
 \sum_{j_1 \neq j_2 \in J} y_{j_1} y_{j_2} \prob{ j_1 \stackrel{B_t}{\longleftrightarrow} j_2 }
+
 2 \sum_{i \in I,\, j \in J} x_{i} y_{j} \prob{ i \stackrel{B_t}{\longleftrightarrow} j }.
\end{multline}
Now note that the number of visits to $I$ of a simple path in $B_t$  is at most $|I|$.
Using this idea and a calculation similar to \eqref{eq_proof_connect_i_j_in_G_t}, we obtain 
 the inequalities
\begin{align*}
\prob{ i_1 \stackrel{B_t}{\longleftrightarrow} i_2 } \;\leq \; &
(x_{i_1}x_{i_2} \cdot b \cdot t^2) \cdot  \sum_{k=1}^{|I|} ( t^2 a b)^{k-1}, \quad
i_1 \neq i_2, \; i_1, i_2 \in I
  \\
\prob{ j_1 \stackrel{B_t}{\longleftrightarrow} j_2 } \; \leq \; &
(y_{j_1}y_{j_2} \cdot a \cdot t^2)  \cdot  \sum_{k=1}^{|I|} ( t^2 a b)^{k-1}, \quad
j_1 \neq j_2, \; j_1, j_2 \in J
  \\
\prob{ i \stackrel{B_t}{\longleftrightarrow} j }\; \leq \; &
(x_{i} y_{j} t)  \cdot  \sum_{k=1}^{|I|} ( t^2 a b)^{k-1}, \quad
i \in I, \; j \in J
\end{align*}
Combining these inequalities with \eqref{eq_bipartite_S_2_expand} 
 we obtain \eqref{eq_bipartite_S_2_B_t_statement_finite_I} as well as
\begin{multline*}
\expect{ S_2^{B_t}} - a \stackrel{\eqref{eq_assumption_bipartite}}{\leq} 
b + 2 \left( a^2 \cdot b \cdot t^2 +
b^2 \cdot a \cdot t^2 +
2 a \cdot b \cdot t
 \right) \stackrel{\eqref{eq_assumption_bipartite}}{\leq} \\
  b \cdot \left( 1 + 2 a^2 t^2 + 1 + 4 a t \right)=
  2 b \left(1 + a t \right)^2.
\end{multline*}
This completes the proof of \eqref{eq_bipartite_S_2_B_t_statement}.
\end{proof}

\begin{lemma}\label{lemma_mc_graphical_rep_cadlag}
With probability $1$, the function $t \mapsto \mathrm{ord}(\um,G_t)$ (see \eqref{eq_def_ord_um_G}) is \cadlag with respect to the 
$\dist(\cdot,\cdot)$-metric (defined in \eqref{eq_def_d_metric}).
\end{lemma}

\begin{proof}
Let us fix some $T \geq 0$. Denote by $A$ the event
\begin{equation}\label{event_A_cadlag_on_this}
A=\{ S_2^{G_T}< + \infty \} \cap 
\left\{
\begin{array}{c}
\text{ for any }  i,j \in \N \text{ the number of  simple }
\\
\text{ paths connecting $i$ and $j$ in $G_T$ is
   finite }
   \end{array}
\right\}
\end{equation}

 By Lemma \ref{lemma:S_2_G_t_as_finite} the event $A$ almost surely holds. 
Assuming that $A$ holds, we will show that $t \mapsto \mathrm{ord}(\um,G_t)$ is \cadlag
on $[0,T)$.

Since $G_s \subseteq G_t$ if $s \leq t$, we can apply Lemma \ref{lemma_dominate_graph_compare_aldous}
 in order to reduce our task to showing that the function $t \mapsto S_2^{G_t}$
is  \cadlag on $[0,T)$.
If $A$ holds, then for any $i,j \in \N$ the function $t \mapsto \ind[ i \stackrel{G_t}{\longleftrightarrow} j ]$  is \cadlag on $[0,T)$.
  Using this fact,  
 \eqref{S2_def} and the dominated convergence theorem, we obtain
that indeed $t \mapsto S_2^{G_t}$ is also \cadlag on $[0,T)$.

\end{proof}

%\begin{definition}\label{def_G_m_downarrow_uparrow}
%For any $\underline{x} \in \ltwopos$ and $t \geq 0$, $m \in \N_+$ 
%define the graph $G^{m \downarrow}_t$  to be the subgraph of $G_t$ spanned by the vertex set $\{1,\dots,m\}$ and 
%$G^{m \uparrow}_t$ to be the subgraph of $G_t$ spanned by the vertex set $\{m+1, m+2,\dots \}$.
%\end{definition}

\section{Graphical construction of MCLD($\lambda$)}
\label{subsection:deletions}

Recall the informal definition of the $\mathrm{MCLD}(\lambda)$ process $\bm_t$ from \eqref{mcld_informal_def}.
We now give a graphical construction of the process $\bm_t$  with initial state $\um \in \ltwoord$ and deletion rate $\lambda$.
 Let 
 \begin{equation}\label{exponetials_xi_lambda}
 \begin{array}{l}
\text{ $\left( \xi_{i,j} \right)_{1 \leq i<j<\infty}$ be  random variables with 
$\mathrm{EXP}(1)$ distribution,} \\
\text{ $\left( \lambda_{i} \right)_{1 \leq i <\infty}$
 be random variables with
$\mathrm{EXP}(\lambda)$ distribution,}
\end{array} 
\end{equation}
and let us also assume that all of these random variables are independent.

The heuristic description of our graphical construction is as follows: we increase $t$ continuously and
if the event
 $\xi_{i,j} = t m_i m_j$ occurs for some $1\leq i<j <\infty$, we merge the components of the vertices $i$ and $j$, moreover if $\lambda_i  = t m_i$ for some $i \in \N_+$, then we say that a \emph{lightning} strikes vertex $i$ and
  delete the connected component of vertex $i$.
 Since the total rate of merger and deletion events is infinite if $\sum_i m_i = +\infty$,
  we need to be  careful with the above heuristic definition if we want to make it precise: we will now provide the graphical construction.
 
 \medskip
  
In Definition \ref{def_mc_graphical} we defined the simple graph $G_t$ with vertex set $\N_+$.

We will define for any $t \in \R_+$ 
\begin{equation}
\begin{array}{l}
\text{the set  of intact vertices $\mathcal{V}_t \subseteq \N_+$  and} \\
\text{the set of burnt vertices $\N_+ \setminus \mathcal{V}_t$.}
\end{array}
\end{equation}
The graph $H_t$ will denote the
 subgraph of $G_t$ spanned by $\mathcal{V}_t$ and $\bm_t$ will denote the ordered sequence of component weights of $H_t$.
 
 \medskip
 
Recall that we enumerated the connected components
$\mathcal{C}_k(t), k \in \N_+$
of $G_t$ in \eqref{eq_def_componnets_of_G_t}.
  By the properties of exponential random variables, \eqref{finite_components} and the independence
   of $\left( \xi_{i,j} \right)_{1\leq i<j <\infty}$ and $\left( \lambda_{i} \right)_{i =1}^{\infty}$,
 we see that for every $t \geq 0$
\begin{equation}\label{finitely_many_lightnings_in_each_component}
\prob{ \forall  k \in \N_+: \, \sum_{i \in \mathcal{C}_k(t)} 
\ind[ \lambda_i  \leq t m_i ] < +\infty  }=1.
\end{equation}

This implies that for every $t \geq 0$ and $k \in \N_+$, there exists an almost surely finite $\N$-valued random variable 
$N$ (the number of lightnings that hit the component $\mathcal{C}_k(t)$ by time $t$), indices $i_1, \dots, i_N \subseteq  \mathcal{C}_k(t)$ (the vertices that are hit by lightning) and times $0<t_1<\dots<t_N \leq t$ (the ordered sequence of the times of the lightnings) such that
\[\{ \; i \in \mathcal{C}_k(t) \; : \; \lambda_i  \leq t m_i \;\} = \{ \;i_1, \dots, i_N \; \}
\quad \text{and} \quad
\forall \, 1 \leq l \leq  N \; : \; t_l = \frac{\lambda_{i_l}}{m_{i_l}}.
  \]

We now define the set  of intact vertices $\mathcal{V}_t \subseteq \N_+$ by constructing
$ \mathcal{V}_t \cap \mathcal{C}_k(t)$ for every $k \in \N_+$. 

Let us fix $k \in \N_+$.
We recursively define $\mathcal{V}_{t_l} \cap \mathcal{C}_k(t)$
for each $1 \leq l \leq  N$ in the following way.

\begin{enumerate}[(i)]
\item  At $t_0=0$ we have 
  $\mathcal{V}_{t_0} \cap \mathcal{C}_k(t)= \mathcal{C}_k(t)$.
 \item Assume that we have already constructed $\mathcal{V}_{t_{l-1}} \cap \mathcal{C}_k(t)$
for some $1 \leq l \leq  N$.
 We define $\mathcal{V}_{t_l} \cap \mathcal{C}_k(t)$ by deleting the connected component 
of $i_l$ in the restriction of the graph $G_{t_{l}}$ to the vertex set $\mathcal{V}_{t_{l-1}} \cap \mathcal{C}_k(t)$.
\item
With this recursion we define $\mathcal{V}_{t_N} \cap \mathcal{C}_k(t)$.
Since there are no lightnings hitting $\mathcal{C}_k(t)$ between $t_N$ and $t$, let
$\mathcal{V}_{t} \cap \mathcal{C}_k(t)= \mathcal{V}_{t_N} \cap \mathcal{C}_k(t)$.
\end{enumerate}

Since $\mathcal{C}_k(t), k \in \N_+$
is a partition of $\N_+$, we define 
\begin{equation}\label{def_V_t_H_t_graphical}
\begin{array}{l}
\text{ $\mathcal{V}_t= \bigcup_{k\geq 1} \left(\mathcal{V}_t \cap \mathcal{C}_k(t)\right)$ and}\\
\text{ $H_t$ to be the subgraph of $G_t$ spanned by $\mathcal{V}_t$.}
\end{array}
\end{equation}

Recalling Definition \ref{def_weights_from_graph} we let
 \begin{equation}\label{def_bm_t_graphical}
  \bm_t={\mathrm{ord}}(\um,H_t).
  \end{equation}

\begin{lemma}\label{lemma_indeed_MCLD}
For any  $\um \in \ltwoord$ the  graphical construction  \eqref{def_bm_t_graphical} of the process
 $\bm_t$    gives an $\mathrm{MCLD}(\lambda)$ process with initial condition 
$\um$, i.e.,  an $\ltwoord$-valued Markov process whose dynamics
 satisfy the informal definition given in \eqref{mcld_informal_def}.
\end{lemma} 
\begin{proof}
 $\bm_t$ is a random element of $\ltwoord$, because we have
\[\Vert \bm_t \Vert_2^2=S_2^{H_t} \leq S_2^{G_t} \stackrel{\eqref{S_2_G_t_as_finite}}{<}+\infty.\]
The fact that $\bm_t$ is a Markov process with the prescribed transition rates follows from the memoryless
property and independence of the random variables $\left( \xi_{i,j} \right)_{1\leq i<j <\infty}$ and $\left( \lambda_{i} \right)_{i =1}^{\infty}$. We omit further details.
\end{proof}

\begin{proof}[Proof of Proposition \ref{lemma_mcld_graphical_rep_cadlag}]
We will show that with probability $1$, the function $t \mapsto \mathrm{ord}(\um,H_t)$ is \cadlag with respect to the 
$\dist(\cdot,\cdot)$-metric, see \eqref{eq_def_d_metric}.

Let us fix some $T \geq 0$. 
We know that the event $A$ defined in \eqref{event_A_cadlag_on_this} almost surely holds.
Denote by $B$ the event that every connected component of $G_T$ is exposed to only finitely many lightning strikes on $[0,T]$.
By \eqref{finitely_many_lightnings_in_each_component}, the event $B$ occurs almost surely.
 Assuming that $A \cap B$ holds, we will show that $t \mapsto \mathrm{ord}(\um,H_t)$ is \cadlag
on $[0,T)$. For any $t \geq 0$, define
\begin{itemize}
\item  $\widehat{H}_{t+\Delta t}$ to be the subgraph of $G_t$ spanned by
 $\mathcal{V}_{t+\Delta t}$,
\item  
$\widecheck{H}_{t+\Delta t}$ to be the subgraph of $G_{t+\Delta t}$ spanned by
 $\mathcal{V}_{t}$.
\end{itemize} 
Recalling \eqref{def_V_t_H_t_graphical} and the inclusions $G_t \subseteq G_{t+\Delta t}$ and 
$\mathcal{V}_{t +\Delta t} \subseteq \mathcal{V}_t$ we see that
 \[\widehat{H}_{t+\Delta t} \subseteq H_{t} \subseteq \widecheck{H}_{t+\Delta t}
 \quad \text{and} \quad
 \widehat{H}_{t+\Delta t} \subseteq H_{t+\Delta t} \subseteq \widecheck{H}_{t+\Delta t},
 \]  so
we can apply Lemma \ref{lemma_dominate_graph_compare_aldous}
 and the triangle inequality in order to reduce our task of proving right-continuity of $t \mapsto \mathrm{ord}(\um,H_t)$ at $t$ to showing that
\[  (\mathrm{a}) \, \lim_{\Delta t \to 0_+ }  S_2^{\widecheck{H}_{t+\Delta t}}-  S_2^{H_t}=0, 
\qquad 
(\mathrm{b}) \lim_{\Delta t \to 0_+ }  S_2^{H_t}- S_2^{\widehat{H}_{t+\Delta t}}=0.
 \]
 Now (a) follows from the fact that the graphical representation of the
  multiplicative coalescent possesses the \cadlag property 
 (see Lemma \ref{lemma_mc_graphical_rep_cadlag}).
 
In order to show (b) we observe that on the event $B$, for every connected component 
$\mathcal{C}$ of $G_T$, we have 
\[\lim_{\Delta t \to 0} \ind[ \, \exists\, i \in \mathcal{C} \; : \;  t m_i < \lambda_i  \leq (t +\Delta t) m_i \, ]=0.\]
Given this observation, we see that for every connected component $\mathcal{C}$ of $H_t$ we have
$\lim_{\Delta t \to 0} \ind[ \, \mathcal{C} \subseteq \mathcal{V}_{t + \Delta t} \, ]=1$. 
Using this fact, $S_2^{H_t}<\infty$ and
the dominated convergence theorem, we obtain (b).
 
 The proof of the existence of left limits is similar and we omit it.
\end{proof}

\section{Feller property of $\mathrm{MCLD}(\lambda)$ }
\label{section:feller_coupling_proof}

\begin{definition} 
The  graphical construction of Section \ref{subsection:deletions} gives a joint realization of all of the $\mathrm{MCLD}(\lambda)$ processes with different initial conditions by using the same collection of random variables
 $\left( \xi_{i,j} \right)_{1\leq i<j <\infty}$ and $\left( \lambda_{i} \right)_{1 \leq i <\infty }$ (see \eqref{exponetials_xi_lambda}).
  We call this coupling the $(\xi, \lambda)$-coupling.
\end{definition}

\begin{theorem}\label{thm:feller}
Let $\um^{(n)}, n \in \N$ be a convergent sequence of elements of  $\ltwoord$ and let $\um^{(\infty)}$
denote their limit, i.e., $\lim_{n \to \infty} \dist(\um^{(n)},\um^{(\infty)})=0 $.
For any $t \in \R_+$ and $n \in \N_+ \cup \{\infty\}$,
 denote by $\bm^{(n)}_t$ the $\mathrm{MCLD}(\lambda)$ process with initial condition 
$\um^{(n)}$ at time $t$. Under the $(\xi, \lambda)$-coupling, we have
\begin{equation}\label{feller_convergence_in_probability}
\dist(\bm^{(n)}_t, \bm^{(\infty)}_t) \stackrel{p}{\longrightarrow} 0, \quad n \to \infty.
\end{equation}
\end{theorem}
Theorem \ref{thm:feller} implies that the $\mathrm{MCLD}(\lambda)$ Markov process
 indeed possesses the Feller property, i.e., Theorem \ref{thm:feller_basic} holds.
 
\medskip

We want to prove Theorem \ref{thm:feller} using truncation, because \eqref{feller_convergence_in_probability}  trivially
holds for the truncated process. However, we cannot directly apply Lemma \ref{lemma_dominate_graph_compare_aldous}
to compare the original with the truncated process, because
 we cannot upper bound the state of the truncated process at time $t$ by the state of the
  original process at time $t$ (c.f.\ Remark \ref{remark_lack_of_monotonicity}).
  
   In Section \ref{subsection_truncation} we overcome this problem by
  introducing two auxiliary objects that upper/lower bound both the original and the truncated object, but
  yet these auxiliary  objects can be shown to be close to each other if we only throw away a small part of the
  original when we truncate.

In Section \ref{subsection_proof_of_feller_for_MCLD} we prove Theorem \ref{thm:feller} using the results of 
Section \ref{subsection_truncation} and variant of the $\varepsilon/3$-argument.

\subsection{Bounding the effect of truncation}
\label{subsection_truncation}

In this subsection, we will fix $t \geq 0$ as well as an initial state $\um \in \ltwoord$,
  and omit the dependence of random variables on $t$ and $\um$. We also fix a truncation threshold $m \in \N$.

\begin{definition}\label{def_G_truncated} 
Recall  Definition \ref{def_mc_graphical}.
Denote by $G$, $G^{m \downarrow}$ and
$G^{m \uparrow}$ the graphs with adjacency matrix $\ind[ \xi_{i,j} \leq t m_i m_j ]$
on the vertex set $\N_+$, $\{1,\dots,m\}$, and $\{m+1,m+2,\dots \}$, respectively.

Let $\um^{(m)}$ denote the vector $\um$ truncated at index $m$: 
\begin{equation*}%\label{truncation_def_single}
 \um^{(m)}=(m_1, \dots, m_m, 0, 0, \dots ), \;
\text{where} \;
\um=(m_1, m_2, \dots ).
\end{equation*}

Let $\bm$ (resp.\ $\bm^{(m)}$) denote the state  at time $t$ of the realization under the $(\xi, \lambda)$-coupling of the $\mathrm{MCLD}(\lambda)$ process with initial state
$\um$ (resp.\ $\um^{(m)}$).

Denote by $\mathcal{V}$
and $\mathcal{V}^{(m)}$ the corresponding sets of intact vertices, see \eqref{def_V_t_H_t_graphical}.

 Denote by $H$ and $H^{(m)}$ the subgraphs of
$G$ spanned by $\mathcal{V}$ and $\mathcal{V}^{(m)}$.

\end{definition}

  In order to compare $\bm$ with  $\bm^{(m)}$, we 
  need the following result.

\begin{lemma}\label{lemma_auxiliary_graph_inclusions}

If   $\widehat{G}^{(m)}$ and $\widecheck{G}^{(m)}$ are random graphs with vertex sets
\[V(\widehat{G}^{(m)}), V(\widecheck{G}^{(m)}) \subseteq \N_+\]  and
 under the  $(\xi, \lambda)$-coupling
we have
\begin{equation}\label{G_nm_H_nm_inclusions}
\widehat{G}^{(m)} \subseteq H^{(m)} \subseteq \widecheck{G}^{(m)}, \quad
\widehat{G}^{(m)} \subseteq H \subseteq \widecheck{G}^{(m)}
\end{equation}
then almost surely we have 
\begin{equation}\label{sandwich}
\dist( \bm, \bm^{(m)}) \leq  3 \cdot \sqrt{S_2^{\widecheck{G}^{(m)}}- S_2^{\widehat{G}^{(m)}}}.
\end{equation}
\end{lemma}

\begin{proof}
First note that it follows from \eqref{G_nm_H_nm_inclusions} that
\begin{equation}\label{G_nm_H_nm_S_2_ineqs}
S_2^{\widehat{G}^{(m)}} \leq S_2^{H^{(m)}} \leq S_2^{ \widecheck{G}^{(m)} },
\quad S_2^{\widehat{G}^{(m)}} \leq S_2^{H} \leq S_2^{ \widecheck{G}^{(m)} }.
\end{equation}
Thus we have
\begin{multline*}
\dist( \bm, \bm^{(m)})
\stackrel{ \eqref{eq_def_ord_um_G} }{\leq} 
\dist( \bm, {\mathrm{ord}}(\um,\widecheck{G}^{(m)} ) )+ \\
 \dist(  {\mathrm{ord}}(\um,\widecheck{G}^{(m)} ),   {\mathrm{ord}}(\um^{(m)},  \widehat{G}^{(m)})   )+ 
\dist( {\mathrm{ord}}(\um^{(m)},  \widehat{G}^{(m)})  , \bm^{(m)}  ) 
\stackrel{(*)}{\leq} \\
\sqrt{ S_2^{ \widecheck{G}^{(m)} } - S_2^{H}}+
\sqrt{S_2^{ \widecheck{G}^{(m)} } - S_2^{\widehat{G}^{(m)}}}+ 
\sqrt{S_2^{H^{(m)}} - S_2^{\widehat{G}^{(m)}}} 
\stackrel{\eqref{G_nm_H_nm_S_2_ineqs}}{\leq} 
 3 \cdot \sqrt{S_2^{\widecheck{G}^{(m)}}- S_2^{\widehat{G}^{(m)}}},
\end{multline*}
where  $(*)$ follows from \eqref{def_bm_t_graphical},
 the inclusions \eqref{G_nm_H_nm_inclusions} and Lemma \ref{lemma_dominate_graph_compare_aldous}.
%This concludes the proof of Lemma \ref{lemma_auxiliary_graph_inclusions}.
\end{proof} 
In Definition \ref{def_sandwich_m_graphs} below we will construct auxiliary graphs $\widehat{G}^{(m)}$ and $\widecheck{G}^{(m)}$
 in such a way that \eqref{G_nm_H_nm_inclusions} holds.
  Recall Definition \ref{def_G_truncated}.
 Note that $H^{(m)}$ is the subgraph of $G^{m \downarrow}$ spanned by the vertex set
 $\mathcal{V}^{(m)}$.
 In particular, every connected component of $H^{(m)}$ is a subset of a connected component of
 $G^{m \downarrow}$. 
 
 The next definition only involves the random variables $\left( \xi_{i,j} \right)_{1\leq i<j <\infty}$  (i.e., we don't have to look at $\left( \lambda_{i} \right)_{i =1}^{\infty}$).
 
 \begin{definition}\label{def_bipartite_parallel}
  Given $G^{m \downarrow}$ and $G^{m \uparrow}$, denote the connected components
 of $G^{m \downarrow}$ by $\mathcal{C}^{m \downarrow}_k, k \in K$ and the
 connected components
 of $G^{m \uparrow}$ by $\mathcal{C}^{m \uparrow}_l, l \in L$. 
 
Let us define an auxiliary bipartite multigraph $\mathcal{B}$ with vertex set $K \cup L$.
Declare $k \in K$ and $l \in L$ connected in $\mathcal{B}$ if $\mathcal{C}^{m \downarrow}_k$ is connected to
$\mathcal{C}^{m \uparrow}_l$ in $G$.  We allow \emph{parallel} edges to be present in $\mathcal{B}$:
if  $\mathcal{C}^{m \downarrow}_k$ is connected to
$\mathcal{C}^{m \uparrow}_l$ by more than one edge in $G$, then we put an equal number of parallel edges between
$k \in K$ and $l \in L$ in $\mathcal{B}$.
\end{definition}
 
 Now we define a subset $K^* \subseteq K$ indexing  ``bad" components of $G^{m \downarrow}$.
  This definition involves the random variables $\left( \xi_{i,j} \right)_{1\leq i<j <\infty}$ as well as $\left( \lambda_{i} \right)_{i =1}^{\infty}$. The components indexed by $k \in K \setminus K^*$ are ``good''. 
  The key property of good components will be stated in Lemma \ref{lemma_good_component_properties} below.

\begin{definition}\label{def_bad_components_feller} 
Recall the definition of $\mathcal{B}$ from 
Definition \ref{def_bipartite_parallel}.

\begin{enumerate}[(i)]
\item An \emph{edge-simple path} in  $\mathcal{B}$ is a path with no repeated edges.
\item We say that  $k \in K$ (resp.\ $l \in L$) is \emph{intact} if no lightning hit any vertex of $\mathcal{C}^{m \downarrow}_k$
(resp.\ $\mathcal{C}^{m \uparrow}_l$) before time $t$. If a vertex of $\mathcal{B}$ is not intact, then we
 say that it is \emph{damaged}.
\item  We say that $k \in K^*$ if $k \in K$ and there is a edge-simple path in $\mathcal{B}$ 
 which consists of at least one edge and
connects $k$ to a damaged vertex of $\mathcal{B}$.
\end{enumerate}
\end{definition}
For an illustration of Definition \ref{def_bad_components_feller}, see Figure \ref{fig:bad_vertices}.

\begin{figure}[h]
\begin{center}
\large
\psfrag{Gd}{$G^{m \downarrow}$}
\psfrag{Gu}{$G^{m \uparrow}$}

  \includegraphics[scale=0.6]{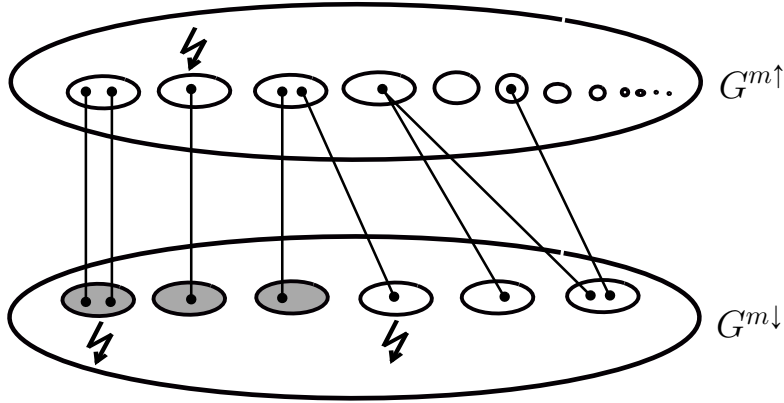}
\caption{An illustration of Definition \ref{def_bad_components_feller}. 
 The blobs marked with a lightning are
 damaged connected components of
$G^{m \downarrow}$ and $G^{m \uparrow}$. 
The grey blobs are the ``bad" components of $G^{m \downarrow}$. The set of indices of ``bad'' components is denoted by $K^*$.
Note that  intact connected components of $G^{m \downarrow}$ can be ``bad'' and  
 damaged connected components of $G^{m \downarrow}$ can be ``good''.
}
\label{fig:bad_vertices}
\end{center}
\end{figure}

%Now we state a useful property of the ``good" components of $G^{m \downarrow}$. 
\begin{lemma} \label{lemma_good_component_properties}
Recalling Definition \ref{def_G_truncated}, we have
\begin{equation}\label{good_components_identity}
\forall \, k \in K \setminus K^* \, : \; \;
\mathcal{C}^{m \downarrow}_k \cap \mathcal{V}^{(m)} = \mathcal{C}^{m \downarrow}_k \cap \mathcal{V}.
\end{equation}
\end{lemma}
\begin{proof} Let $k \in K \setminus K^*$.
 Denote by $\mathcal{C}'$ the connected component of $k$ in $\mathcal{B}$.
We prove \eqref{good_components_identity} by considering two cases separately.

\smallskip

{\bf First case:} $k$ is intact. 

Denote by $K'= \mathcal{C}' \cap K$ and $L'=\mathcal{C}' \cap L$. 
Then \[ \mathcal{C}=\left(\bigcup_{k' \in K'} \mathcal{C}^{m \downarrow}_{k'}\right) \cup 
\left(\bigcup_{l' \in L'} \mathcal{C}^{m \uparrow}_{l'} \right) \]
is a connected component of $G$  which contains $\mathcal{C}^{m \downarrow}_{k}$ (c.f.\ Definition \ref{def_bipartite_parallel}), moreover our assumption that $k$ is intact together with $k \in K \setminus K^*$ imply that $\mathcal{C}$ is intact (c.f.\ Definition \ref{def_bad_components_feller}), thus we have 
$\mathcal{C}^{m \downarrow}_k \cap \mathcal{V}^{(m)} = \mathcal{C}^{m \downarrow}_k$ and 
$\mathcal{C}^{m \downarrow}_k \cap \mathcal{V}=\mathcal{C}^{m \downarrow}_k$,
  therefore \eqref{good_components_identity} holds.

\smallskip

{\bf Second case:} $k$ is damaged.

$\mathcal{C}' \setminus \{k\}$ is the disjoint union of some connected components $\mathcal{C}'_N, N \in \N$ of 
$\mathcal{B} \setminus \{k\}$.
 Our assumption that $k$ is damaged, Definition \ref{def_bad_components_feller} and the fact that
$k \in K \setminus K^*$ together imply that there are no parallel edges connected to $k$ in $\mathcal{B}$ and no edge-simple circle of the graph $\mathcal{B}$ contains $k$ as a vertex. Therefore  for each $N \in \N$, the cluster $\mathcal{C}'_N$ is connected to $k$ by one single edge $e_N$ of $\mathcal{B}$. Note that $k \in K \setminus K^*$ implies that $\mathcal{C}'_N$ is intact for all $N \in \N$. Therefore, the fires caused by lightnings can only spread ``away'' from $k$ on the edges $e_N, N \in \N$,
so by the graphical construction given in Section \ref{subsection:deletions} and Definition \ref{def_G_truncated}
 we obtain \eqref{good_components_identity}.

\end{proof}

 Now we define auxiliary random graphs  $\widehat{G}^{(m)}$ and $\widecheck{G}^{(m)}$ (c.f.\ Lemma \ref{lemma_auxiliary_graph_inclusions}). 
 
 \begin{definition}\label{def_sandwich_m_graphs}
 Let  $\widecheck{G}^{(m)}$  be the subgraph of $G$ spanned by the vertices
 \begin{equation}\label{eq_def_major_G_m}
  V(\widecheck{G}^{(m)})=
 \left( \bigcup_{ k \in K \setminus K^* } 
 \mathcal{C}^{m \downarrow}_k \cap \mathcal{V}^{(m)}   \right) \cup
 \left( \bigcup_{ k \in K^* } \mathcal{C}^{m \downarrow}_k \right) \cup 
  \{ m+1,m+2,\dots \} .
 \end{equation}
 Define  $\widehat{G}^{(m)}$ to be the subgraph of $G$ spanned by the vertices
 \begin{equation}\label{eq_def_minor_G_m}
  V(\widehat{G}^{(m)})=
   \bigcup_{ k \in K \setminus K^* } 
 \mathcal{C}^{\downarrow m}_k \cap \mathcal{V}^{(m)}.
  \end{equation}
\end{definition}  
  
\begin{lemma} With the above definitions the inclusions \eqref{G_nm_H_nm_inclusions} hold. \end{lemma}
\begin{proof}
 The inclusions $V(\widehat{G}^{(m)}) \subseteq \mathcal{V}^{(m)} \subseteq  V(\widecheck{G}^{(m)})$  follow
 from the definitions \eqref{eq_def_major_G_m}, \eqref{eq_def_minor_G_m}.
Thus 
  $\widehat{G}^{(m)} \subseteq H^{(m)} \subseteq \widecheck{G}^{(m)}$ follows from  the fact that $H^{(m)}$
is the subgraph of $G$ spanned by the vertex set $\mathcal{V}^{(m)}$.

 The inclusions $\widehat{G}^{(m)} \subseteq H \subseteq \widecheck{G}^{(m)}$
follow from  Lemma \ref{lemma_good_component_properties} 
and the fact that $H$
is the subgraph of $G$ spanned by the vertex set $\mathcal{V}$.
\end{proof}

 The  next lemma is similar to Lemma \ref{lemma_bipartite}.
\begin{lemma}\label{lemma_minor_major_close}
 Given the above set-up let us condition on the graphs $G^{m \downarrow}$ and $G^{m \uparrow}$ and 
 denote by
\[ \alpha = S_2^{G^{m \downarrow}}, \quad \beta = S_2^{G^{m \uparrow}}. \]
 There exists a constant $C=C(\lambda,t)$ such that if
 \begin{equation}\label{t_2_alpha_beta_leq_half}
  t^2 \alpha \beta \leq \frac12 
  \end{equation}
 holds then we have
\begin{equation}\label{ineq_minor_major_close}
\condexpect{S_2^{\widecheck{G}^{(m)}}- S_2^{\widehat{G}^{(m)}}}{ G^{m \downarrow}, \; G^{m \uparrow} } \leq 
C \cdot \beta \cdot \left( (1 + t \alpha)^2 + (1 + t \alpha) \cdot \alpha^{3/2} \right).
\end{equation}
\end{lemma}

\subsubsection{Proof of  Lemma \ref{lemma_minor_major_close}}
\label{subsection_proof_lemma_min_maj}

 % We fixed $t \in \R_+$ and $\lambda \in \R_+$.
% Recall Definition \ref{def_G_truncated}.
% Note that $H^{(m)}$ is the subgraph of $G^{m \downarrow}$ spanned by the vertex set
% $\mathcal{V}^{(m)}$.
% In particular, every connected component of $H^{(m)}$ is subset of a connected component of
% $G^{m \downarrow}$. 
 
%Recall the bipartite multigraph $\mathcal{B}$ from Definition  \ref{def_bipartite_parallel} and the set of ``bad" vertices $K^*$ from
%Definition \ref{def_bad_components_feller}.
%Recall the property \eqref{good_components_identity} of ``good" components.

%  $\widecheck{G}^{(m)}$ is the subgraph of $G$ spanned by the vertices  $V(\widecheck{G}^{(m)})$, see
 %  \eqref{eq_def_major_G_m}.
   
 %  $\widehat{G}^{(m)}$ is the subgraph of $G$ spanned by the vertices  $V(\widehat{G}^{(m)})$, see 
 %  \eqref{eq_def_minor_G_m}.
 
% Given the above set-up we conditioned on the graphs $G^{m \downarrow}$ and $G^{m \uparrow}$ and 
% introduced the notation
%\[ \alpha = S_2^{G^{m \downarrow}}, \quad \beta = S_2^{G^{m \uparrow}}. \]
% In order to prove Lemma \ref{lemma_minor_major_close} we need to show that there
%  exists a constant $C=C(\lambda,t)$ such that if
% \begin{equation}\label{t_2_alpha_beta_leq_half}
%  t^2 \alpha \beta \leq \frac12 
%  \end{equation}
 %holds, then we have \eqref{ineq_minor_major_close}.

 For any subset $\mathcal{C}$ of $\N$, denote by 
\[w(\mathcal{C})= \sum_{i \in \mathcal{C}} m_i\] 
the weight
of the subset, where $\um=(m_1,m_2,\dots)$.

\begin{definition}
Define a bipartite weighted graph $\widetilde{\mathcal{B}}$ whose "left" vertices correspond to
 the connected components of the restriction of  $G$  to the vertex set
 \begin{equation*} \widetilde{V}^{(m)}:= 
 V(\widecheck{G}^{(m)}) \cap \{1,\dots,m\} \stackrel{ \eqref{eq_def_major_G_m} }{=}
 \left( \bigcup_{ k \in K \setminus K^* } 
 \mathcal{C}^{m \downarrow}_k \cap \mathcal{V}^{(m)}   \right) \cup
 \left( \bigcup_{ k \in K^* } \mathcal{C}^{m \downarrow}_k \right) ,
\end{equation*} 
  and the "right"
 vertices correspond to the components of $G^{m \uparrow}$. Define the weights of the vertices of 
 $\widetilde{\mathcal{B}}$ to be the $w(\cdot)$-weight of the corresponding connected components.
  We declare two vertices in $\widetilde{\mathcal{B}}$ 
 to be connected
 if the corresponding subsets  are connected in $\widecheck{G}^{(m)}$.
 Denote by $\widetilde{G}^{(m)}$ the subgraph of $G$ spanned by $\widetilde{V}^{(m)}$.
\end{definition} 
 
 With the above notation we have
 \[ S_2^{\widecheck{G}^{(m)}} \stackrel{ \eqref{eq_def_major_G_m} }{=} S_2^{\widetilde{\mathcal{B}}}, \qquad 
  S_2^{\widetilde{G}^{(m)}}\stackrel{ \eqref{eq_def_minor_G_m} }{=}
 S_2^{\widehat{G}^{(m)}} + \sum_{k \in K^*} w(\mathcal{C}^{m \downarrow}_k)^2.
 \]
Thus we can start to rewrite the left-hand side of \eqref{ineq_minor_major_close}:
\begin{multline*}
 \condexpect{S_2^{\widecheck{G}^{(m)}}- S_2^{\widehat{G}^{(m)}}}{ G^{m \downarrow}, \; G^{m \uparrow} } =\\
 \condexpect{ S_2^{\widetilde{\mathcal{B}}} - S_2^{\widetilde{G}^{(m)}}}{ G^{m \downarrow}, \; G^{m \uparrow} }+
 \condexpect{ \sum_{k \in K^*} w(\mathcal{C}^{m \downarrow}_k)^2 }{ G^{m \downarrow}, \; G^{m \uparrow} }.
  \end{multline*}
In order to show \eqref{ineq_minor_major_close}, it is enough to prove that \eqref{t_2_alpha_beta_leq_half}
implies
\begin{align}
\label{truncation_uniform_expect_bound_1}
\condexpect{ S_2^{\widetilde{\mathcal{B}}} - S_2^{\widetilde{G}^{(m)}}}{ G^{m \downarrow}, \; G^{m \uparrow} } &\leq
2 \beta \cdot  (1 + t \alpha)^2,  \\
\label{truncation_uniform_expect_bound_2}
\condexpect{ \sum_{k \in K^*} w(\mathcal{C}^{m \downarrow}_k)^2 }{ G^{m \downarrow}, \; G^{m \uparrow} } &\leq
2 t^2 \lambda \beta \cdot (1 + t \alpha) \cdot \alpha^{3/2}.
\end{align}
First we deduce \eqref{truncation_uniform_expect_bound_1} from Lemma \ref{lemma_bipartite}, with the
underlying bipartite graph being  $\widetilde{\mathcal{B}}$. Note that
the condition \eqref{eq_assumption_bipartite} holds, because 
$a=S_2^{\widetilde{G}^{(m)}} \leq  S_2^{G^{m \downarrow}}= \alpha$ and $b= S_2^{G^{m \uparrow}}=\beta $.
 Thus we have
\begin{multline*}
\condexpect{ S_2^{\widetilde{\mathcal{B}}} - S_2^{\widetilde{G}^{(m)}}}
{ G^{m \downarrow}, \; G^{m \uparrow}, \; (\lambda_i)_{i=1}^m } 
\stackrel{\eqref{eq_bipartite_S_2_B_t_statement}}{\leq}\\
2 S_2^{G^{m \uparrow}} \cdot  (1 + t S_2^{\widetilde{G}^{(m)}} )^2 \leq
2 \beta \cdot  (1 + t \alpha)^2.
\end{multline*}
Now \eqref{truncation_uniform_expect_bound_1} follows by averaging over the values of $(\lambda_i)_{i=1}^m$.

\smallskip

In order to prove \eqref{truncation_uniform_expect_bound_2}, we first 
give an upper bound on the probability of the event  $\{ k \in K^*\}$.

For $k \in K$, denote by $x'_k = w(\mathcal{C}^{m \downarrow}_k)$
and for $l \in L$, denote $y'_l=  w(\mathcal{C}^{m \uparrow}_{l})$. Note that we have
\[  \alpha= \sum_{k \in K} (x'_k)^2, \qquad \beta= \sum_{l \in L} (y'_l)^2. \]
Recall the definition of $K^*$ from Definition \ref{def_bad_components_feller}.
The next calculation is similar to \eqref{eq_proof_connect_i_j_in_G_t}, so we omit the 
first few steps.
\begin{multline*}
\condprob{ k \in K^* }{ G^{m \downarrow}, \; G^{m \uparrow} } \leq \\
\sum_{ l_1 \in L} \left( x'_k y'_{l_1} t \right)
\left( \lambda y'_{l_1} t \right) + 
\sum_{ l_1 \in L} \sum_{ k_1 \in K} \left( x'_k y'_{l_1} t \right)
\left( y'_{l_1} x'_{k_1} t \right)
\left( \lambda x'_{k_1} t \right) + \\
\sum_{ l_1 \in L} \sum_{ k_1 \in K} \sum_{l_2 \in L}
 \left( x'_k  y'_{l_1} t \right)
\left( y'_{l_1} x'_{k_1} t \right) 
\left( x'_{k_1} y'_{l_2} t \right)
 \left( \lambda y'_{l_2} t \right) + \dots  = \\
 x'_k t^2 \lambda \beta + x'_k t^3 \lambda \alpha \beta + 
  x'_k t^4 \lambda \alpha \beta^2 + \dots = \\
  x'_k t^2 \lambda \beta \cdot \left( 1 + t \alpha \right) \cdot
 \sum_{n=0}^{\infty} \left( t^2 \alpha \beta \right)^n 
 \stackrel{\eqref{t_2_alpha_beta_leq_half}}{\leq} 2  x'_k t^2 \lambda \beta \cdot \left( 1 + t \alpha \right).
 \end{multline*}
Now we are ready to prove \eqref{truncation_uniform_expect_bound_2}:
\begin{multline*}
\condexpect{ \sum_{k \in K^*} (x'_k)^2 }{ G^{m \downarrow}, \; G^{m \uparrow} }\leq
\sum_{k \in K} 2  (x'_k)^3 t^2 \lambda \beta \cdot \left( 1 + t \alpha \right) 
\stackrel{(*)}{\leq}\\
2 t^2 \lambda \beta \cdot (1 + t \alpha) \cdot \alpha^{3/2},
\end{multline*}
where in $(*)$ we used the fact that $x'_k \leq \sqrt{\alpha}$ for any $k \in K$.
This completes the proof of \eqref{ineq_minor_major_close} and Lemma \ref{lemma_minor_major_close}.

\subsection{Proof of Theorem \ref{thm:feller}}
\label{subsection_proof_of_feller_for_MCLD}

Let us fix $t, \lambda \in \R_+$, the sequence $\um^{(n)}, n \in \N$ and the limit $\um^{(\infty)}$.
 For any $n \in \N_+ \cup \{ \infty \}$,
let $\bm^{(n,m)}_t$ denote the realization under the $(\xi, \lambda)$-coupling of the
 $\mathrm{MCLD}(\lambda)$ with initial state
\begin{equation}\label{truncation_def}
 \um^{(n,m)}=(m^{(n)}_1, \dots, m^{(n)}_m, 0, 0, \dots ), \;
\text{where} \;
\um^{(n)}=(m^{(n)}_1, m^{(n)}_2, \dots ).
\end{equation}
We also define $\mathcal{V}^{(n,m)}_t$ to be the set of intact vertices of the graph
 $H^{(n,m)}_t$ of the $\mathrm{MCLD}(\lambda)$ with initial state $\um^{(n,m)}$ under the $(\xi, \lambda)$-coupling.

In order to prove \eqref{feller_convergence_in_probability} we only need to show that
for every $\varepsilon > 0$ there exists $m, n_0 \in \N$ such that for all $n \geq n_0$ we have
\begin{align}
\label{X^n_and_X^nm_close}
\prob{ \dist( \bm^{(n)}_t, \bm^{(n,m)}_t) \geq \varepsilon } &\leq 4\varepsilon, \\
\label{X^nm_and_X^inftym_close}
\prob{ \dist( \bm^{(n,m)}_t, \bm^{(\infty ,m)}_t) \geq \varepsilon } &\leq \varepsilon, \\
\label{X^inftym_and_X^infty_close}
\prob{ \dist( \bm^{(\infty,m)}_t, \bm^{(\infty)}_t) \geq \varepsilon } &\leq 4\varepsilon.
\end{align}

 Let us fix  $\varepsilon>0$.
  We know from Lemma \ref{lemma:S_2_G_t_as_finite} that
 \begin{equation*}%\label{as_finite_S_G_infty_t}
  \prob{S_2^{G^{(\infty)}_t}<+\infty}=1,
  \end{equation*}
 where $G^{(\infty)}_t$ denotes the random graph constructed from the exponential variables
 $\left( \xi_{i,j} \right)_{1\leq i<j <\infty}$ and the initial state $\um^{(\infty)} \in \ltwoord$ according to the rules
 described in Definition \ref{def_mc_graphical}. Given $\varepsilon>0$, we can find $M \in \R_+$ such that
  \begin{equation} \label{eq_def_M}
   \prob{S_2^{G^{(\infty)}_t}\geq M-1} \leq \varepsilon.
  \end{equation}
Recall the notion of the constant $C=C(t, \lambda)$ from  Lemma \ref{lemma_minor_major_close}.
Let us choose $\delta>0$  such that
 \begin{equation}\label{choice_of_delta}
t^2 M \delta \leq \frac12 \quad \text{and} \quad
 9 C \cdot \delta \cdot \left( (1 + t M)^2 + (1 + t M) \cdot M^{3/2} \right) \leq \varepsilon^3.
 \end{equation}
  Now we choose the truncation threshold $m$. Since $\um^{(n)} \to \um^{(\infty)}$ in $l_2$, we can make
 \[\sup_{n \in \N \cup \{ \infty \} } \Vert \um^{(n)}- \um^{(n,m)} \Vert_2 \]
 (where $\um^{(n,m)}$ is defined in \eqref{truncation_def})
  as small as we wish by making $m$ large. Thus by \eqref{small_S2_if_small_initial_S2}
and the Markov inequality we can choose $m$ such that
\begin{equation}\label{remaining_small_S2}
 \sup_{n \in \N \cup \{ \infty \} }
\prob{S_2^{G_t^{(n,m) \uparrow}} \geq \delta} \leq \varepsilon.
\end{equation}  
 Having fixed $m$, we note that under the $(\xi, \lambda)$-coupling
 we have
 \begin{equation*}
\dist(\bm^{(n,m)}_t, \bm^{(\infty,m)}_t) \stackrel{p}{\longrightarrow} 0, \quad n \to \infty.
\end{equation*}
 We also have
\begin{equation}\label{limsup_truncated_S2}
 S_2^{G^{(n,m) \downarrow }_t} \stackrel{p}{\longrightarrow} S_2^{G^{(\infty,m) \downarrow }_t} \leq S_2^{G^{(\infty)}_t},
\end{equation}
 thus we can choose $n_0$ such that for all $n \geq n_0$ we have \eqref{X^nm_and_X^inftym_close} and
 \begin{equation}\label{uniform_truncated_S2}
 \forall \; n \in \{n_0, n_0+1,\dots \} \cup \{ \infty \} \; : \; 
 \prob{S_2^{G^{(n,m)\downarrow }_t}\geq M} \stackrel{\eqref{eq_def_M}, \eqref{limsup_truncated_S2}}{\leq } 2\varepsilon.
 \end{equation}
We are ready to show  \eqref{X^n_and_X^nm_close} and \eqref{X^inftym_and_X^infty_close} for the above choice of $m$ and $n_0$.

For any 
$n \in  \{n_0, n_0+1, \dots \} \cup \{ \infty \}$ we have
\begin{multline*}
\prob{ \dist( \bm^{(n)}_t, \bm^{(n,m)}_t) \geq \varepsilon} \leq
\prob{S_2^{G_t^{(n,m) \uparrow}} \geq \delta}+
\prob{S_2^{G^{(n,m) \downarrow}_t}\geq M} + \\
\prob{ \dist( \bm^{(n)}_t, \bm^{(n,m)}_t) \geq \varepsilon, \; 
 S_2^{G_t^{(n,m) \uparrow}} \leq \delta, \, S_2^{G^{(n,m)\downarrow }_t}\leq M }
  \stackrel{\eqref{remaining_small_S2}, \eqref{uniform_truncated_S2}}{\leq}\\
\varepsilon + 2 \varepsilon+
\prob{ \dist( \bm^{(n)}_t, \bm^{(n,m)}_t) \geq \varepsilon, \; A },
 \end{multline*}
where $A=\{ S_2^{G_t^{(n,m) \uparrow}} \leq \delta, \, S_2^{G^{(n,m)\downarrow }_t}\leq M \}$. We bound
\begin{multline*}
\prob{ \dist( \bm^{(n)}_t, \bm^{(n,m)}_t) \geq \varepsilon, \; A } \stackrel{\eqref{sandwich}}{\leq}
\prob{  9 \cdot \left(S_2^{\widecheck{G}^{(n,m)}}- S_2^{\widehat{G}^{(n,m)}}\right) \geq \varepsilon^2, \; A }=
\\
\expect{ 
\condprob{  9 \cdot \left(S_2^{\widecheck{G}^{(n,m)}}- S_2^{\widehat{G}^{(n,m)}}\right) \geq \varepsilon^2 }
 { G^{(n,m) \downarrow}, \; G^{(n,m) \uparrow} }\; ; \; A 
 } 
 \\
 \stackrel{(*)}{\leq} 
 \frac{9 C \cdot \delta \cdot \left( (1 + t M)^2 + (1 + t M) \cdot M^{3/2} \right)}{\varepsilon^2}
 \stackrel{\eqref{choice_of_delta}}{\leq}  \varepsilon,
\end{multline*} 
where in the equation marked by $(*)$ we used
 Lemma \ref{lemma_minor_major_close} and the Markov inequality.
This concludes the proof of \eqref{X^n_and_X^nm_close},\eqref{X^nm_and_X^inftym_close},\eqref{X^inftym_and_X^infty_close}
and Theorem \ref{thm:feller}.
 
\bigskip

 {\bf Acknowledgements:} I thank James Martin for collaborating with me on \cite{rigid_paper}, which
 inspired this work. I also thank an anonymous referee for useful comments on the manuscript.
  This work is partially supported by OTKA (Hungarian
National Research Fund) grant K100473, the Postdoctoral Fellowship of
the Hungarian Academy of Sciences and the Bolyai Research Scholarship
of the Hungarian Academy of Sciences.

\end{document}